\documentclass[11pt]{amsart}
\usepackage{latexsym}
\usepackage{rotating}
\usepackage{amsfonts}
\usepackage{amssymb,amscd}
\usepackage{amsmath}
\usepackage{txfonts}
\usepackage{balance}
\usepackage{textcase}
\usepackage{float}
\usepackage{amsthm}
\usepackage{geometry}
\usepackage[colorlinks, linkcolor=blue, citecolor=black]{hyperref}
\newtheorem{theorem}{Theorem}[section]
\newtheorem{lemma}[theorem]{Lemma}
\newtheorem{definitiona}[theorem]{\textbf{Definition}}

\theoremstyle{remark}

\newtheorem{remark}[theorem]{Remark}

\usepackage{courier}

\usepackage{graphicx}
\graphicspath{ {./graphics/} }

\usepackage{hyperref}
\newtheorem{corollarya}[theorem]{\textup{\textbf{Corollary}}}
\newtheorem{prop}{\textup{\textbf{Proposition}}}
\newtheorem*{terminology}{\textup{\textbf{Terminology } }\hspace{1mm}}
\newtheorem*{question}{\textup{\textbf{Question}}}
\theoremstyle{remark}
\numberwithin{equation}{section} \theoremstyle{plain}

\newcommand{\thmref}[1]{Theorem~\ref{#1}}
\newcommand{\lemref}[1]{Lemma~\ref{#1}}

\newcommand{\propref}[1]{Proposition~\ref{#1}}
\newcommand{\corollaryaref}[1]{Corollary~\ref{#1}}

\numberwithin{equation}{section}

\begin{document}
	
	\title{A Limit Set Intersection Theorem for Graphs of Relatively Hyperbolic Groups}
	
	\author[Swathi Krishna]{Swathi Krishna} 
	\address{Indian Institute of Science Education and Research (IISER) Mohali, Knowledge City, 
		Sector 81,  S.A.S. Nagar, Punjab 140306, India}
	\email{swathi280491@gmail.com}
	
	\maketitle

	\begin{abstract} 
		Let $G$ be a relatively hyperbolic group that admits a decomposition into a finite graph of relatively hyperbolic groups structure with quasi-isometrically (qi) embedded condition. We prove that the set of conjugates of all the vertex and edge groups satisfy the limit set intersection property for conical limit points (refer to Definition \ref{def:1} and Definition \ref{def:2} for the definitions of conical limit points and limit set intersection property respectively). This result is motivated by the work of Sardar for graph of hyperbolic groups \cite{PS}.
	\end{abstract}

	\section{Introduction}\label{sec1} 
	
	Limit set intersection  theorem first appeared in the work of Susskind \cite{Sus}, in the context of geometrically finite subgroups of Kleinian groups. Later, Susskind and Swarup \cite{SusSwa} proved it for geometrically finite purely hyperbolic subgroups of a discrete subgroup of Isom($\mathbb{H}^n$). Here, a hyperbolic group is a discrete subgroup of Isom$(\mathbb{H}^n)$. A group $G$ is geometrically finite if there is a finite sided fundamental polyhedron for the action of $G$ on $\mathbb{H}^3$. The works of Susskind and Swarup were followed by the work of J.W. Anderson in \cite{JWA1}, \cite{JWA2} and \cite{JWA3} for some classes of subgroups of Kleinian groups. Susskind asked the following question: 
	
	\begin{question}
		Let $\Gamma$ be a non-elementary Kleinian group acting on $\mathbb{H}^n$ for some $n \geq 2$, and let $H, K$ be non-elementary subgroups of $\Gamma$, then is $\Lambda_{c}(H) \cap \Lambda_{c}(K) \subset \Lambda(H \cap K)$ true? Here $\Lambda_{c}(H)$ and $\Lambda_{c}(K)$ denote the conical limit sets of $H$ and $K$ (see Definition \ref{def:1}) respectively.
	\end{question} 
	
	In an attempt to answer this, Anderson showed that if $\Gamma$ is a non-elementary purely loxodromic Kleinian group acting on $\mathbb{H}^n$ for some $n\geq 2$ (cf.\cite{JWA4}) and $H$ and $K$ are non-elementary subgroups of $\Gamma$, then $\Lambda_{c}(H) \cap \Lambda_{c}^{u}(K) \subset \Lambda_{c}(H \cap K)$, where $\Lambda_{c}^{u}(K)$ denotes the uniform conical limit sets of $K$. But in \cite{TDDS}, Das and Simmons constructed a non-elementary Fuchsian group $\Gamma$ that admits two non-elementary
	subgroups $H, K \leq \Gamma$ such that $H \cap K = \{e\}$ but $\Lambda_{c}(H) \cap \Lambda_{c}(K) \neq \emptyset$, thus providing a negative answer to Susskind's question.
	
	However, this prompts the following question in the context of hyperbolic and relatively hyperbolic groups:
	
	\begin{question}
		Suppose $\Gamma$ is a hyperbolic (resp. relatively hyperbolic) group and $H, K$ are subgroups of $\Gamma$, then is $\Lambda_{c}(H) \cap \Lambda_{c}(K) \subset \Lambda(H \cap K)$ true?
	\end{question}
	
	In 2012, Yang \cite{Yang} proved a limit set intersection theorem for relatively quasiconvex subgroups of relatively hyperbolic groups. Limit set intersection theorem is not true for general subgroups of hyperbolic groups, and it was known to hold only for quasiconvex subgroups until the recent work of Sardar \cite{PS}. In the paper, he claimed that a limit set intersection theorem holds for limit sets of vertex and edge subgroups of a graph of hyperbolic groups, however, in communication with Sardar it has been pointed out that this only holds for conical limit sets, see \cite{PS1}. We generalize this to relatively hyperbolic graph of groups in the following theorem:
	\begin{theorem}\label{thm1}
		Let $G$ be a group admitting a decomposition into a finite graph of relatively hyperbolic groups $(\mathcal{G},Y)$  satisfying the qi-embedded condition. Further, suppose the monomorphisms from edge groups to vertex groups is strictly type-preserving, and that induced tree of coned-off spaces also satisfy the qi-embedded condition. If $G$ is hyperbolic relative to the family $\mathcal{C}$ of maximal parabolic subgroups, then the set of conjugates of vertex and edge groups of $G$ satisfy a limit set intersection property for conical limit points.
	\end{theorem}
	
	The proof relies heavily on the ladder construction by Mj and Pal in \cite{MjPal}. The terminology is briefly recalled in section \ref{sec4}
	
	\textbf{Outline of the paper:}
	
	\begin{enumerate}
		
		\item First we recall the construction of the tree of relatively hyperbolic metric spaces associated to a graph of relatively hyperbolic groups in \ref{sec4}
		
		\item We give a modified construction of the ladder from \cite{MjPal} in \ref{sec5}
		
		\item Using \propref{prop1}, \lemref{lem12} and \lemref{lem13}, we prove that if boundary points of two vertex spaces are mapped to the same point under the Cannon Thurston map (see Definition \ref{CT}), such that the image is a conical limit point for each of these vertex spaces, then these boundary points can be flowed (see Definition \ref{flow}) to each other.
		
		Finally, we prove \thmref{thm1} in \ref{finalsec}
		
	\end{enumerate}
	
	
	\section{Preliminaries}\label{sec2}
	
	\subsection{Hyperbolic spaces}
	
	For definitions and basic properties of hyperbolic metric spaces, hyperbolic groups and its boundary one may refer to \cite{MBAH} and \cite{IKNB}.
	For a quick review of limit points and results pertaining to it, one may refer to \cite{PS}. 
	
	\begin{definitiona}\label{def:1.1} \it \textbf{Limit set} (see \textup{\cite{PS}})\textbf{:}
		{\it Let $X$ be a hyperbolic metric space and let $Y$ be a subset of $X$. The limit set of $Y$, denoted by $\Lambda(Y)$, is defined to be $\Lambda(Y) = \{\xi \in \partial{X}\,|\, \exists \, \{y_n\} \subset Y \text{\textup{ st }} \lim{y_n} \to \xi\}$.} 
	\end{definitiona}
	
	Let $(X,d)$ be a metric space. For a subset $Y \subset X$ and $R >0$, $N_{R}(Y) = \{x \in X\, |\, \exists \, y \in Y \text{ with } d(x,y) \leq R \}$.  
	\begin{definitiona}\label{def:1}\it \textbf{Conical limit point}\textbf{:}
		{\it \begin{enumerate}
				\item Let $X$ be a proper hyperbolic metric space and $Y \subset X$. Then $\xi \in \partial{X}$ is called a {\em conical limit point} of $Y$ if for any geodesic ray $\gamma$ in $X$ asymptotic to $\xi$, there is a constant $R < \infty$ such that, there exists sequence $\{y_n\}$ in $Y \cap N_{R}(\gamma)$ with  $\lim{y_n} \to \xi$.
				
				\item For a group $H$ acting on $X$ by isometries, $\xi \in \Lambda(H)$ is a {\em conical limit point} of $H$ if $\xi$ is a conical limit point of the orbit $H\cdot{x_0}$ for any $x_0 \in X$. 
				
				\item The set of all conical limit points of $H$ is called the {\em conical limit set} and it is denoted by $\Lambda_{c}(H)$. 
	\end{enumerate}} \end{definitiona}
	
	The first two parts of Definition \ref{def:1} also make sense for an infinite subgroup or subset $H$ of a hyperbolic group $G$. In that case, we may take $X$ to be a Cayley graph of $G$ and the action of $H$ on $X$. We state two results on conical limit set of any such $H$, which easily follow from \cite[Lemma 1.2]{PS}.
	\begin{lemma}\label{lem1}
		Suppose $G$ is a hyperbolic group and let $H$ be a subset of $G$. Then for every $g \in G$,
		\begin{enumerate}
			\item $\Lambda_{c}(gHg^{-1}) = \Lambda_{c}(gH)$;
			\item $\Lambda_{c}(gH) = g\Lambda_{c}(H)$.
		\end{enumerate}
	\end{lemma}
	
	This also holds for the set of non-conical limit points. Let $\Lambda_{nc}(H) := \Lambda(H) \setminus {\Lambda_{c}(H)}$ denote the set of non-conical limit points. Then clearly $\Lambda_{nc}(gHg^{-1}) = \Lambda_{nc}(gH)$ and $\Lambda_{nc}(gH) = g\Lambda_{nc}(H)$.
	
	\begin{definitiona}\label{CT}\it \textbf{Cannon-Thurston map:}
		{\it Let $X, Y$ be proper hyperbolic metric spaces and let $f: X \to Y$ be a proper embedding. A {Cannon-Thurston} (CT) {map} $\bar{f}: \overline{X} \to \overline{Y}$ is a continuous extension of $f$.}
		
		Here, $\overline{X}= X \cup \partial{X}$ and $\overline{Y}= Y \cup \partial{Y}$, i.e., their respective visual compactifications. We denote $\overline{f}|_{\partial{X}}$ by $\partial{f}$. 
	\end{definitiona}
	
	By \cite[Lemma 2.6]{PS}, if Cannon Thurston map exists for the map  $f: X \to Y$, where $f, X, Y$ are as in the above definition, then $\partial{f}(\partial{X}) = \Lambda(f(X))$.
	
	Let $X$ be a geodesic metric space. For any $x, y \in X$, a geodesic joining $x$ and $y$ is denoted by $[x,y]$. The following are two basic results of $\delta$-hyperbolic metric spaces.
	
	
	\begin{lemma}\label{lem7}\textup{\cite{Mj}} 	
		Given $\delta >0$, there exists $D$, $C_1$ such that if $a, b, c, d$ are vertices of a $\delta$-hyperbolic metric space $(X,d)$, with $d(a,[b,c]) = d(a,b)$, $d(d,[b,c]) = d(c,d)$ and $d(b,c) \geq D$ then $[a,b] \cup [b,c] \cup [c,d]$ lies in a $C_1$-neighbourhood of any geodesic joining $a$ and $b$.
	\end{lemma}
	
	\begin{lemma}\label{lem8}\textup{\cite{Mj}}
		Let $(X,d)$ be $\delta$-hyperbolic metric space. Let $\gamma$ be geodesic in $Y$ and $x \in X$. Let $y$ be a nearest point projection of $x$ on $\gamma$. Then for any $z \in \gamma$, a geodesic path from $x$ to $y$ followed by a geodesic path from $y$ to $z$ is a $k$-quasigeodesic, for some $k=k(\delta)$.
	\end{lemma}     .
	\subsection{Relatively hyperbolic spaces}\label{3}
	
	Relative hyperbolic groups was introduced by Gromov in his article \cite{Gromov} on hyperbolic groups. Gromov \cite{Gromov}, Farb \cite{Farb} and Bowditch \cite{Bowditch} provide good reference to the various notions of relative hyperbolicity. 
	
	We briefly recall some important definitions and basic results here.
	
	\begin{definitiona}\label{conedoff}\it \textbf{Coned-off space} (see \textup{\cite{Farb}})\textbf{:}
		{\it Let $(X,d)$ be a path metric space and $\mathcal{A}= \{A_\alpha\}_{\alpha \in \Lambda}$ be a collection of uniformly separated subsets of $X$, i.e., there exists $\epsilon >0$ such that $d(A_\alpha,A_\beta)> \epsilon$ for all distinct $A_\alpha, A_\beta$ in $\mathcal{A}$. For each $A_\alpha \in \mathcal{A}$, introduce a vertex $\nu(A_\alpha)$ and join every element of $A_\alpha$ to the vertex by an edge of length $\frac{1}{2}$. This new space is denoted by $\widehat{X} = \mathcal{E}(X,\mathcal{A})$. The new vertices are called cone points and $H_\alpha \in \mathcal{H}$ are called horosphere-like sets. The new space is called a {\em coned-off space of $X$ with respect to $\mathcal{A}$}}.
	\end{definitiona}	
	
	\begin{terminology} 
		\begin{enumerate}
			\item Let $X$ be a geodesic metric space. For $x,y \in X$, $d(x,y)$ or $d_{X}(x,y)$ denotes the distance in the original metric on $X$. For any two subsets $A,B \subset X$, we denote the Hausdorff distance between them by $\textup{Hd}(A,B)$. For $C \geq 0$, $N_{C}(A)$ will denote the $C$-neighbourhood of $A$ in $X$.
			
			\item The induced length metric on $\widehat{X}$ is called the \textbf{electric metric}.
			\item For a geodesic metric space $(X,d)$, let $\widehat{X}$ denote the coned-off metric space relative to a collection of horosphere-like sets $\{A_\alpha\}_{\alpha \in \Lambda}$. Then for $x,y \in \widehat{X}$, $d_{\widehat{X}}(x,y)$ denotes the distance in the electric metric.
			\item Geodesics and quasigeodesics in $\widehat{X}$ are called \textbf{electric geodesics} and \textbf{electric quasigeodesics} respectively.
			\item Let $\gamma$ be a path in $X$. If $\gamma$ penetrates a horosphere-like set $A_\alpha$, we replace portions of $\gamma$ inside $A_\alpha$ by edges joining the entry and exit points of $\gamma$ in $A_\alpha$ to $\nu(A_\alpha)$. We denote the new path by $\hat{\gamma}$. If $\hat{\gamma}$ is an electric geodesic (resp. electric quasi-geodesic), we call $\gamma$ a \textbf{relative geodesic} (resp. \textbf{relative quasigeodesic}) in $X$.
			\item For any electric geodesic $\hat{\alpha}$, we denote the union of subsegments of $\hat{\alpha}$ lying outside the horosphere-like sets by $\alpha^b$. 
			\item A path $\gamma$ in $X$ is a path \textbf{without backtracking} if it does not return to any coset $A_\alpha$ after leaving it. 
			
		\end{enumerate} 
	\end{terminology}
	
	\begin{definitiona}\it \textbf{Bounded region penetration property:}
		{\it Let $(X,\mathcal{A})$ be as in Definition \ref{conedoff}. The pair $(X,\mathcal{A})$ satisfies {bounded region penetration property} if, for every $K \geq 1$, there exists $B = B(K)$ such that if $\beta$ and $\gamma$ are two relative $K$-quasi-geodesics without backtracking and joining the same pair of points, then
			
			\begin{enumerate}
				\item if $\beta$ penetrates a horosphere-like set $A_\alpha$ and $\gamma$ does not, then the length of the portion of $\beta$ lying inside $A_\alpha$ is at most $B$, with respect to the metric on $X$;
				
				\item if both $\beta$ and $\gamma$ penetrate a horosphere-like set $A_\alpha$, then the distance between the entry points of $\beta$ and $\gamma$ into $A_\alpha$ and the distance between the exit points of $\beta$ and $\gamma$ from $A_\alpha$ is at most $B$, with respect to the metric on $X$.
			\end{enumerate}
		}
	\end{definitiona}
	
	\begin{definitiona}\it \textbf{Strongly relative hyperbolic space}(see \textup{\cite{Farb}})\textbf{:}
		{\it A metric space $X$ is {\em strongly hyperbolic relative to a collection of subsets $\mathcal{A}$} if the coned-off space $\mathcal{E}(X, \mathcal{A})$ is a hyperbolic metric space and $(X,\mathcal{A})$ satisfies the bounded region penetration property.}
	\end{definitiona}
	
	\begin{definitiona}\it \textbf{Strongly relative hyperbolic group}(see \textup{\cite{Farb}})\textbf{:}
		{\it A group $G$ is {\em strongly hyperbolic relative to a collection of subgroups $\mathcal{H} = \{H_{\alpha}\}_{\alpha \in \Lambda}$} if the Cayley graph $X$, of $G$, is strongly hyperbolic relative to the collection of subgraphs corresponding to the left cosets of $H_{\alpha}$ in $G$ for every $\alpha \in \Lambda$.}
	\end{definitiona}
	
	Another equivalent definition of relatively hyperbolic groups that we use is due to Gromov.
	
	\begin{definitiona}\it \textbf{Hyperbolic cone} (see \textup{\cite{Gromov}})\textbf{:}
		{\it Let $(Y,d)$ be a geodesic space. Then the {\em hyperbolic cone} of $Y$, $Y^h = Y \times[0,\infty)$ with the path metric $d_h$ is defined as follows:
			
			\begin{enumerate}
				\item For $(x,t), (y,t) \in Y \times\{t\}$, $d_{h,t}((x,t),(y,t))= e^{-t}d(x,y)$, where $d_{h,t}$ is the induced path metric on $Y \times \{t\}$. Paths joining $(x,t)$ and $(y,t)$ that lie in $Y \times [0,\infty)$ are called horizontal paths.
				\item For $t, s \in [0,\infty)$ and any $x \in Y$, $d_{h}((x,t),(x,s))= |t-s|.$ Paths joining such elements are called vertical paths.
			\end{enumerate}
			In general, for $x, y \in Y^h$, $d_{h}(x,y)$ is the path metric induced by these vertical and horizontal paths.}
	\end{definitiona}
	
	\begin{definitiona}\it \textbf{Relatively hyperbolic space} (see \textup{\cite{Gromov}})\textbf{:}
		{\it Let $X$ be a geodesic metric space and $\mathcal{A}$ be a set of mutually disjoint subsets. For each $A\in\mathcal{A}$, we attach a hyperbolic cone $A^h$ to $A$ by identifying $(x,0)$ with $x$ for all $x \in A$. This space is denoted by $X^h = \mathcal{G}(X,\mathcal{A})$. $X$ is said to be {\em hyperbolic relative to} $\mathcal{A}$ in the sense of Gromov if $\mathcal{G}(X,\mathcal{A})$ is a complete hyperbolic space.}
	\end{definitiona}
	
	\begin{definitiona}\it \textbf{Relatively hyperbolic group} (see \textup{\cite{Gromov}})\textbf{:}
		{\it Let $G$ be a finitely generated group and $\mathcal{H}=\{H_\alpha\}_{\alpha \in \Lambda}$ be a collection of finitely generated subgroups. Let $\Gamma$ be the Cayley graph of $G$ and let $H_{(g,\alpha)}$ be the subgraph corresponding to the left coset $gH_{\alpha}$ in $\Gamma$. We denote it by  $\Gamma^h = \mathcal{G}(\Gamma,\{H_{(g,\alpha)}\}_{\alpha \in \Lambda, g \in G})$. $G$ is said to be \textbf{hyperbolic relative to $\mathcal{H}$} in the sense of Gromov, if $\Gamma^h$ is a complete hyperbolic metric space.}
	\end{definitiona}
	
	\begin{terminology}
		\begin{enumerate}
			\item For a geodesic metric space $(X,d)$, let ${X}^h$ denote the metric space with hyperbolic cones attached to the collection of horosphere-like sets. Then for $x,y \in {X}^h$, $d_{X^h}(x,y)$ denotes the distance in the path metric of $X^h$. For any two subsets $A,B \subset X^h$, we denote the Hausdorff distance between them by $\textup{Hd}_{X^h}(A,B)$.
			\item For $C \geq 0$, $N_{C}^h(Z)$ will denote a $C$-neighbourhood of a subset $Z$ of $(X^h,d_{X^h})$.
			
			\item A geodesic (resp. quasigeodesic) in $X^h$ is called a \textbf{hyperbolic geodesic} (resp. \textbf{hyperbolic quasigeodesic}).
			
			\item Let $\hat{\alpha}$ be an electric quasigeodesic without backtracking in $\widehat{X}$. For each $A_\alpha$ penetrated by $\hat{\alpha}$, let $x, y$ be the entry and exit points of $\hat{\alpha}$, respectively. We join $x$ and $y$ by a geodesic in $A^h_{\alpha}$. This gives a path in $X^h$ and we call it an \textbf{electro-ambient quasigeodesic}. This path is, in fact, a quasigeodesic in $X^h$.
			\item The electro-ambient quasigeodesic corresponding to an electric geodesic $\hat{\alpha}$ is always denoted by $\alpha$.
			\item Let $G$ be hyperbolic relative to a collection of subgroups $\{H_{\alpha}\}$. Let $\Gamma$ denote a Cayley graph of $G$. Then $H_\alpha$ and their conjugates are called \textbf{parabolic subgroups}. In $\Gamma^h$, each hyperbolic cone has a single limit point in $\partial{\Gamma}^h$ and it is called a \textbf{parabolic limit point}.
			
		\end{enumerate} 
	\end{terminology}
	
	\begin{remark}\label{rmk1}
		Suppose a metric space $X$ is strongly hyperbolic relative to a collection of subsets $\mathcal{A}$, then the space obtained by coning off the hyperbolic cones, $\mathcal{E}(\mathcal{G}(X,\mathcal{A}),\mathcal{A}^h)$, is quasi-isometric to $\mathcal{E}(X,\mathcal{A})$. $\mathcal{E}(X,\mathcal{A})$ is isometrically embedded in $\mathcal{E}(\mathcal{G}(X,\mathcal{A}),\mathcal{A}^h)$ and $\mathcal{E}(\mathcal{G}(X,\mathcal{A}),\mathcal{A}^h)$ lies in a 1-neighbourhood of the image of $\mathcal{E}(X,\mathcal{A})$.
	\end{remark}
	
	\begin{lemma}\label{lem5}\textup{\cite[Lemma 1.2.31]{Pal}} Let $K \geq 1$, $\lambda \geq 0$, $\epsilon > 0$, $r \geq 0$. Suppose $X_1, X_2$ are geodesic spaces and $\mathcal{H}_{X_1}, \mathcal{H}_{X_2}$ are collections of $\epsilon$-separated and intrinsically geodesic closed subspaces of $X_1, X_2$ respectively. Let $\phi: X_1 \to X_2$ be a $(K,\lambda)$-quasi-isometry such that for each $H_1 \in \mathcal{H}_{X_1}$, there exists $H_2 \in \mathcal{H}_{X_2}$ such that $\textup{Hd}(\phi(H_1),H_2) \leq r$ in $X_2$ and $\textup{Hd}(\phi^{-1}(H_2),H_1) \leq r$ in $X_1$. Then $\phi: X_1 \to X_2$ induces a $(K^h,\lambda^h)$-quasi-isometry ${\phi}^h: X_{1}^h \to X_{2}^h$, for some $K^h \geq 1$, ${\lambda}^h \geq 0$.
	\end{lemma}
	
	\begin{lemma}\label{lem2}\textup{\cite[Lemma 1.2.19]{Pal}}
		Let $X$ be a geodesic metric space hyperbolic relative to a collection of uniformly $\epsilon$-separated, uniformly properly embedded closed subsets, in the sense of Gromov. Then $X$ is properly embedded in $X^h$ i.e., for all $M > 0$, there exists $N=N(M)$ such that $d_{X^h}(i(x),i(y)) \leq M$ implies $d(x,y) \leq N$, for every $x,y \in X$. Here $i :X \to X^h$ is the inclusion map.
	\end{lemma}
	
	Using \lemref{lem2}, we prove the following result.
	
	\begin{lemma}\label{lem4}
		Let $X$ be a geodesic metric space hyperbolic relative to a collection of uniformly $\epsilon$-separated, uniformly properly embedded closed subsets $\mathcal{A} = \{A_{\alpha}\}_{\alpha \in \Lambda}$, in the sense of Gromov. Let $\gamma$ be a geodesic ray in $X^h$ such that $\gamma(\infty)$ is not a parabolic limit point. Then for any $R>0$, if $x \in X$ such that $x \in N_{R}^h(\gamma)$, then there exists $R_{1} = R_{1}(R)$ such that $x \in N_{R_1}(\gamma \cap X)$.
	\end{lemma}
	
	\begin{proof}
		Let $y \in \gamma$ such that $d_{X^h}(x,y) \leq R$. If $y \in \gamma \cap X$, by \lemref{lem2}, there exists $N_1 = N(R)$ such that $d_{X}(x,y) \leq N_1$. Now, suppose $y \in \gamma \cap A_{\alpha}^h$, for some $\alpha \in \Lambda$. Let $\gamma_1$ denote the geodesic segment $\gamma|_{[a,b]}$, where $a$ denotes the entry point of $\gamma$ into $A_\alpha$ and $b$ denotes the exit point of $\gamma$ from $A_\alpha$. Let $t \in [0,\infty)$ such that for $(a,t), (b,t) \in A_\alpha \times\{t\}$, $d_{h,t}((a,t),(b,t))= e^{-t}d_{A_\alpha}(a,b) = 1$, where $d_{h,t}$ is the induced path metric on $A_\alpha \times \{t\}$. 
		
		Then, $d_{A_\alpha}(a,b) = e^t$ and $t = \ln{d_{A_\alpha}(a,b)}$. Let $\lambda_1$ and $\lambda_2$ denote the vertical paths in ${A}_{\alpha}^h$ joining $(a,0)$ to $(a,t)$ and $(b,0)$ to $(b,t)$ respectively. Let $\lambda_0$ denote the horizontal path in ${A}_{\alpha}^h$ joining $(a,t)$ to $(b,t)$. The path $\lambda = \lambda_1\ast\lambda_0\ast\lambda_2$ is a quasigeodesic in ${A}_{\alpha}^h$ and by stability of quasigeodesics, there exists $K_1 >0$ such that $\textup{Hd}_{X^h}(\gamma_1,\lambda) \leq K_1$. Since $y \in \gamma_1$, there exists $z \in \lambda$ such that $d_{X^h}(y,z) \leq K_1$ and we have, $d_{X^h}(x,z) \leq R+ K_1$. But length of the quasigeodesic $\lambda$ is $2t+1$ and clearly, $t \leq R + K_1$ and $d_{X^h}((a,0),z) \leq t +1 \leq R + K_1 + 1$.  Thus, $d_{X^h}((a,0),x) \leq 2(R+K_1)+1$. By \lemref{lem2}, there exists $N_2 = N(2(R+K_1)+1)$ such that $d_{X}(x,y) \leq N_2$. 
		
		For $R_1 = \max\{N_1,N_2\}$, we have $x \in N_{R_1}(\gamma)$.	
	\end{proof}
	
	\begin{definitiona}\it \textbf{Electric projection}\textup{(see \cite{MjPal})}\textbf{:}
		{\it Let $Y$ be a space hyperbolic relative to the collection $\{A_\alpha\}_{\alpha \in \Lambda}$. Let $i : Y^h \to  \mathcal{E}(\mathcal{G}(Y,\mathcal{A}),\mathcal{A}^h)$ be the inclusion map. we identify $\mathcal{E}(\mathcal{G}(Y,\mathcal{A}),\mathcal{A}^h)$ with $\widehat{Y}$. Let $\hat{\alpha}$ be an electric geodesic in $\widehat{Y}$ and $\alpha$ be the corresponding electro-ambient quasigeodesic. Let ${\pi}_{\alpha}$ be a nearest point projection from $Y^h$ onto $\alpha$.
			Electric projection is the map $\hat{\pi}_{\hat{\alpha}} : \widehat{Y} \to \hat{\alpha}$ given by:
			
			$\text{For } x\in Y, \quad \hat{\pi}_{\hat{\alpha}}(x) = i({\pi}_{\alpha}(x))$.
			
			If $x$ is a cone point of a horosphere like set $A_{\beta} \in \mathcal{A}$, choose some $z \in A_{\beta}$ and define  $\hat{\pi}_{\hat{\alpha}}(x) = i({\pi}_{\alpha}(z)).$}
	\end{definitiona}
	
	\begin{lemma}\label{lem11}\textup{\cite[Lemma 1.16]{MjPal}}
		Let $Y$ be hyperbolic relative to $\mathcal{A}$. There exists a constant $P$ depending upon $\delta$, $D$, $C_1$ such that for any $A \in \mathcal{A}$ and $x,y \in A$ and a geodesic ${\hat{\alpha}}$ in $\widehat{Y}$, then $d_{\widehat{Y}}(i({\pi}_{\alpha}(x)),i({\pi}_{\alpha}(y)) \leq P$.
	\end{lemma}
	This implies that the electric projection is coarsely well-defined. 
	The following theorem, due to Bowditch, gives the equivalence between the two definitions of relative hyperbolicity: 
	
	\begin{theorem}\label{thm3}\textup{\cite{Bowditch}}
		The following are equivalent:
		\begin{enumerate}
			\item X is hyperbolic relative to the collection of uniformly separated subsets $\mathcal{A}$ in $X$.
			\item X is hyperbolic relative to the collection of uniformly separated subsets $\mathcal{A}$ in $X$ in the sense of Gromov.
			\item $X^h$ is hyperbolic relative to the collection $\mathcal{A}^h$.
		\end{enumerate}
	\end{theorem} 
	
	\vspace{2mm}
	
	For a proper hyperbolic metric space $(Y,d)$, we can associate a topological space to it, i.e., its Gromov boundary $\partial{Y}$. Bowditch generalized the Gromov boundary for hyperbolic groups, to the context of relatively hyperbolic groups. 
	\begin{definitiona}\it \textbf{Bowditch boundary:}
		{\it Suppose $X$ is a metric space hyperbolic relative to a collection of subsets $\{A_\alpha\}_{\alpha \in \Lambda}$. Then the {\em Bowditch boundary} (or relative hyperbolic boundary) of $X$ with respect to $\{A_\alpha\}_{\alpha \in \Lambda}$ is the boundary of $X^h$, and it is denoted by $\partial{X}^h$.}
	\end{definitiona}
	
	So, for a group $G$ hyperbolic relative to a collection of subgroups $\mathcal{H}$, its boundary is the boundary of $\Gamma^h$, where $\Gamma$ is a Cayley graph of $G$. 
	
	We end this section with the following definitions of limit set intersection property for relatively hyperbolic groups:
	
	\begin{definitiona}\label{def:2}\it \textbf{Limit set intersection property:}
		{\it Suppose $G$ is a relatively hyperbolic group. Let $\mathcal{S}$ be a collection of subgroups of $G$. Then $\mathcal{S}$ is said to have the {\em limit intersection property} if for every $H$, $K \in \mathcal{S}$, $\Lambda(H) \cap \Lambda(K) = \Lambda(H \cap K)$.}
	\end{definitiona}
	
	\begin{definitiona}\it \textbf{Conical limit intersection property:}
		{\it A collection $\mathcal{S}$ of subgroups of a relatively hyperbolic group $G$ is said to have the {\em conical limit intersection property} if for every $H$, $K \in \mathcal{S}$, $\Lambda_{c}(H) \cap \Lambda_{c}(K) = \Lambda_{c}(H \cap K)$.}
	\end{definitiona}

	
	\section{Graph of groups}\label{sec4}
	
	We briefly recall some definitions related to the graph of groups. One may refer to \cite{Serre} and \cite{SW} for more details. 
	
	\begin{definitiona}\label{graph}\it \textbf{Graph} (see \textup{\cite{Serre}})\textbf{:}
		{\it A graph $Y$ is an ordered pair of sets $(V,E)$ with $V = V(Y)$, the set of vertices of $Y$ and a set $E = E(Y)$, the set of edges of $Y$, and a pair of maps 
			
			$E \to V \times V \hspace{9mm} e \mapsto (o(e),t(e))$;  \hspace{3mm}        and \hspace{3mm}		
			$E \to E \hspace{9mm} e \mapsto \bar{e}$
			
			satisfying the following conditions: $o(\bar{e})= t(e)$, $t(\bar{e})= o(e)$ and $\bar{\bar{e}}= e$ for all $e \in E$. Here, $o(e)$ is the initial vertex of the edge $e$ and $t(e)$ is the terminal vertex; $\bar{e}$ is the inverse of $e$, i.e., the edge $e$ with the opposite orientation.}
	\end{definitiona}
	
	\begin{definitiona}~
		\begin{enumerate}
			\item \it \textbf{Graphs of groups:} 	{\it A graph of groups $(\mathcal{G},Y)$ consists of a finite graph $Y$ with vertex set $V$ and egde set $E$ and for each vertex $v \in V$, there is a group $G_v$ (vertex group) and for each edge $e \in E$, there is a group $G_e$ (edge group), along with the monomorphisms:
				
				$\phi_{o(e)}:G_e \to G_{o(e)}$
				
				$\phi_{t(e)}:G_e \to G_{t(e)}$
				
				with the extra condition that $G_{\bar{e}}=G_e$.}
			
			\item \it \textbf{Graphs of relatively hyperbolic groups:}{\it A graph of groups $(\mathcal{G},Y)$ is a graph of relatively hyperbolic groups if for each $v \in V(Y)$, $G_v$ is hyperbolic relative to a collection of subgroups $\{H_{v,\alpha}\}_{\alpha}$ and for each $e \in E(Y)$, $G_e$ is hyperbolic relative to a collection of subgroups $\{H_{e,\alpha}\}_{\alpha}$.}
		\end{enumerate}
	\end{definitiona}
	
	\begin{definitiona}(see \textup{\cite{MjR}})~
		
		\begin{enumerate}
			\item \it \textbf{QI-embedded condition:} 
			{\it A graph of groups $(\mathcal{G},Y)$ is said to satisfy the {\em qi-embedded condition} if for every $e \in E(Y)$, the monomorphisms $\phi_{o(e)}$ and $\phi_{t(e)}$ are QI-embeddings.
				
				\item \textbf{Strictly type-preserving:} A graph of relatively hyperbolic groups is {\em strictly type-preserving} if for every $e \in E(Y)$, each $\phi_{o(e)}^{-1}(H_{v,\alpha})$ and $\phi_{t(e)}^{-1}(H_{v,\alpha})$ is either empty or some $H_{e,\alpha}$.
				
				\item \textbf{QI-preserving electrocution condition:} $(\mathcal{G},Y)$ satisfies {\em QI-preserving electrocution condition} if induced maps $\hat{\phi}_{o(e)}: \widehat{\Gamma}_{e} \to \widehat{\Gamma}_{o(e)}$ and $\hat{\phi}_{t(e)}: \widehat{\Gamma}_{e} \to \widehat{\Gamma}_{t(e)}$ are uniform qi-embeddings. Here, $\widehat \Gamma_{e}$, $\widehat \Gamma_{o(e)}$ and $\widehat \Gamma_{t(e)}$ denote the coned-off Cayley graphs of $G_e$, $G_{o(e)}$ and $G_{t(e)}$ respectively relative to the corresponding horosphere-like sets.} 
		\end{enumerate}
	\end{definitiona}
	
	\begin{definitiona}\it \textbf{Fundamental group} (see \textup{\cite{Serre}})\textbf{:} 
		{\it Let $(\mathcal{G},Y)$ be a graph of groups. Let $T$ be a maximal subtree of $Y$. Then the fundamental group $G = \pi_{1}(\mathcal{G},Y,T)$ of $(\mathcal{G},Y)$ is defined in terms of generators and relators as:
			
			The generating set is the disjoint union of generating sets of the vertex groups $G_v$ and the set $E(Y)$ of oriented edges of $Y$.
			
			Relators are the following:
			\begin{itemize}
				\item {relators from the vertex groups;}
				\item $\bar{e}= e^{-1};$
				\item $e\phi_{t(e)}(g)e^{-1}= \phi_{o(e)}(g) \,\text{ for all edge } e\textnormal{ and } g\in G_e;$
				\item $e= 1 \text{ if } e \in E(T).$
		\end{itemize}}	
	\end{definitiona}
	
	\begin{definitiona}\it \textbf{Bass-Serre tree of a graph groups}(see \textup{\cite[Section 5.3, Section 5.4]{Serre}})\textbf{:}
		{\it Let $(\mathcal{G},Y)$ be a graph of groups defined above and $G$ be its fundamental group. The Bass-Serre tree is the tree ${\mathcal{T}}$ with vertex set $\bigsqcup_{v \in V(Y)}{G/G_v}$ and edge set $\sqcup_{e \in E(Y)}{G/{G^{e}_e}}$. Here, $G_{e}^e = \phi_{t(e)}(G_e)< G_{t(e)}$. 
			
			So, for an edge $gG^{e}_e$, $o(gG^{e}_e)= gG_{o(e)}$ and $t(gG^{e}_e)= geG_{t(e)}$.}
	\end{definitiona}
	
	Now we give a construction of trees of relatively hyperbolic metric spaces associated to a graph of relatively hyperbolic groups.
	
	\textbf{Trees of relatively hyperbolic metric spaces from a graph of relatively hyperbolic graph of groups} (see \textup{\cite{MjPal}, \cite{MjR},  \cite{PS}})\textbf{:}

	{\it Let $Y$ be a finite graph and $(\mathcal{G},Y)$ be a graph of relatively hyperbolic groups. Let $T$ be a maximal subtree of $Y$ and $G = \pi_1(\mathcal{G},Y,T)$ be the fundamental group of $(\mathcal{G},Y)$. For each $v \in V(Y)$, let $G_v$ be the vertex group hyperbolic relative to $\mathcal{H}_v = \{H_{v,\alpha}\}_{\alpha}$ and for each $e \in E(Y)$, let $G_e$ be the edge group hyperbolic relative to $\mathcal{H}_e = \{H_{e,\alpha}\}_{\alpha}$. For $v \in V(Y)$, we fix the generating set of $G_v$ to be $S_v$ and $e \in E(Y)$, we fix the generating set of $G_e$ to be $S_e$ satisfying $\phi_{t(e)}(S_e) \subset S_{t(e)}$. Then
		$S = \bigcup_{v \in V(Y)}S_{v} \bigcup (E(Y) \setminus E(T))$ is a generating set of $G$. Let $\Gamma(G,S)$ denote the Cayley graph of $G$ with respect to $S$.  	
		
		A tree of relatively hyperbolic metric spaces $X$ for $(\mathcal{G},Y)$ is a metric space admitting a map $p: X \to \mathcal{T}$ and satisfying the following:	
		\begin{enumerate}
			\item For every vertex $\tilde{v} = gG_v \in V(\mathcal{T})$, $X_{\tilde{v}} = p^{-1}(\tilde{v})$ is a subgraph of $\Gamma(G,S)$ with $V(X_{\tilde{v}}) = gG_v$ and $gx, gy \in X_{\tilde{v}}$ are connected by an edge if $x^{-1}y \in S_v$. With the induced path metric $d_{\tilde{v}}$, $X_{\tilde{v}}$ is a geodesic metric space hyperbolic relative to $\mathcal{H}_{\tilde{v}} = \{ gg_{v,\alpha}H_{v,\alpha} | \, g_{v,\alpha}H_{v,\alpha} \, \text{is a left coset of } H_{v,\alpha} \, \text{in } G_v \}$.   
			
			\item For every edge $\tilde{e} = gG_{e}^{e} \in E(\mathcal{T})$, $X_{\tilde{e}} = p^{-1}(\tilde{e})$ is a subgraph of $\Gamma(G,S)$ with $V(X_{\tilde{e}}) = geG_{e}^{e}$ and $gex, gey \in X_{\tilde{e}}$ are connected by an edge if $x^{-1}y \in \phi_{t(e)}(S_e)$. With the induced path metric $d_{\tilde{e}}$, $X_{\tilde{e}}$ is a geodesic metric space hyperbolic relative to $\mathcal{H}_{\tilde{e}} = \{ gg_{e,\alpha}H_{e,\alpha} | \, g_{e,\alpha}H_{e,\alpha} \, \text{is a left coset of } H_{e,\alpha} \, \text{in } G_e \}$. 
			
			\item For an edge $\tilde{e} = gG_{e}^{e}$ connecting vertices $\tilde{u} = gG_{o(e)}$ and $\tilde{v} = geG_{t(e)}$, if $x \in G_{e}^{e}$, we join $gex \in X_{\tilde{e}}$ to $gexe^{-1} \in X_{\tilde{u}}$ and $gex \in X_{\tilde{v}}$ by edges of length $\frac{1}{2}$. 
			These extra edges give us maps $f_{\tilde{e},\tilde{u}} : X_{\tilde{e}} \to X_{\tilde{u}}$ and $f_{\tilde{e},\tilde{v}} : X_{\tilde{e}} \to X_{\tilde{v}}$ with $f_{\tilde{e},\tilde{u}}(gex) = gexe^{-1}$ and $f_{\tilde{e},\tilde{v}}(gex) = gex$.
			
			\item There exists a $\delta > 0$ such that $\mathcal{E}(X_{\tilde{v}},\mathcal{H}_{v})$ and $\mathcal{E}(X_{\tilde{e}},\mathcal{H}_{e})$ are $\delta$-hyperbolic metric spaces.
			
	\end{enumerate}}
	
	A tree of relatively hyperbolic metric spaces $p: X \to \mathcal{T}$ satisfies qi-embedded condition if 
	the maps $f_{\tilde{e},\tilde{u}} : X_{\tilde{e}} \to X_{\tilde{u}}$ and $f_{\tilde{e},\tilde{v}} : X_{\tilde{e}} \to X_{\tilde{v}}$ are qi-embeddings. Further, strictly type-preserving is satisfied if $f_{\tilde{e},\tilde{v}}^{-1}(H_{\tilde{v},\alpha})$ is either some $H_{\tilde{e},\beta} \in \mathcal{H}_{\tilde{e}}$ or empty and for every $H_{\tilde{e},\alpha} \in \mathcal{H}_{\tilde{e}}$, there exists $v$ and $H_{\tilde{v},\beta}$ such that $f_{\tilde{e},\tilde{v}}(H_{\tilde{e},\alpha}) \subset H_{\tilde{v},\beta}$.  
	
	For a tree of relatively hyperbolic metric spaces with vetrex spaces $X_{\tilde{v}}$ and edge spaces $X_{\tilde{e}}$, we can associate a tree of coned-off metric spaces with vertex spaces $\mathcal{E}(X_{\tilde{v}},\mathcal{H}_{v})$ and edge spaces $\mathcal{E}(X_{\tilde{e}},\mathcal{H}_{e})$. This is called the \textbf{induced tree of coned-off spaces}. We denote it by $\mathcal{TC}(X)$. The maps $f_{\tilde{e},\tilde{u}} : X_{\tilde{e}} \to X_{\tilde{u}}$ and $f_{\tilde{e},\tilde{v}} : X_{\tilde{e}} \to X_{\tilde{v}}$ induce $\hat{f}_{\tilde{e},\tilde{u}} : \mathcal{E}(X_{\tilde{e}},\mathcal{H}_{\tilde{e}}) \to \mathcal{E}(X_{\tilde{u}},\mathcal{H}_{\tilde{u}})$ and $\hat{f}_{\tilde{e},\tilde{v}} : \mathcal{E}(X_{\tilde{e}},\mathcal{H}_{\tilde{e}}) \to \mathcal{E}(X_{\tilde{v}},\mathcal{H}_{\tilde{v}})$. If these induced maps are qi-embeddings, then this tree of spaces satisfies qi-preserving electrocution condition.
	
	Now we recall the following from \cite{PS}: Let $v_0 \in V(Y)$ be fixed. Then, $G_{v_0} \in V(\mathcal{T})$. Let $x_0 \in X_{v_0}$ denote the identity element of $G_{v_0}$. By Milnor-Schwarz lemma, the orbit map $\Theta: G \to X$ given by $g \mapsto gx_0$ is a quasi-isometry.
	
	\begin{remark}~
		\begin{enumerate}
			\item There exists a constant $D_0$ such that for every vertex space $gG_v \subset X$, $\textup{Hd}(\Theta(gG_{v}),gG_{v}) \leq D_0$ (cf.\cite[Lemma 3.5]{PS}). For any $gg' \in gG_v$, $\Theta(gg') = gg'x_0$. Let $x$ denote the identity element in $G_v$. Suppose $\gamma_v$ be a geodesic joining $x_0$ to $x$ in $X$. Then $gg'\gamma_{v}$ is a path joining $gg'x_0$ to $gg'x$ in $X$, for every $g' \in G_v$. We choose $D_0 = \mathrm{max}\{l(\gamma_v) \,| \, v \in V(Y)\}$. 
			
			\item Let $\tilde{v}=gG_v \in V(\mathcal{T})$. $\Theta$ induces a quasi-isometry $\Theta_{g,v}: gG_v \to X_{\tilde{v}}$. For each $x \in gG_v$, we map $x$ to $y \in X_{\tilde{v}}$ such that $d_{X}(\Theta(x),y) \leq D_0$. This map is coarsely well-defined. $\Theta$ induces a quasi-isometry $\Theta^h : G^h \to X^h$ and $\Theta_{g,v}$ induces a quasi-isometry $\Theta_{g,v}^h : gG_{v}^h \to X_{\tilde{v}}^h$.
		\end{enumerate}
	\end{remark} 
	
	\begin{definitiona}\it \textbf{Cone locus:}
		{\it The cone locus of $\mathcal{TC}(X)$ is defined as a graph with the vertex set consisting of cone points in the vertex spaces, $\{c_v\,|\, v \in V(\mathcal{T})\}$ and the edge set consists of the cone points in the edge spaces, $\{c_e\,|\, e \in E(\mathcal{T})\}$ . For $u,v \in V(\mathcal{T})$, $c_u$ and $c_v$ are joined by an edge $c_e$, for $e \in E(\mathcal{T})$ if $o(e) = u$, $t(e) = v$ in $\mathcal{T}$, $c_u$, $c_v$ and $c_e$ are cone vertices attached to horosphere-like sets $H_u$ in $\widehat{X}_u$, $H_v$ in $\widehat{X}_v$ and $H_e$ in $\widehat{X}_e$ respectively, and $f_{e,u}(H_{e}) \subset H_{u}$ and $f_{e,v}(H_{e}) \subset H_{v}$. Then the edge $c_{e} \times [0,1]$ joins $c_{u}$ and $c_{v}$ by identifying $c_{e} \times \{0\}$ to $c_{u}$ and $c_{e} \times \{1\}$ to $c_{v}$. } 
	\end{definitiona}
	
	It is easy to see that the connected components of a cone locus are trees. Corresponding to each such connected component, we get a tree of horosphere-like subsets in $X$. We denote the collection of such tree of horosphere-like sets by $\mathcal{C} = \{C_\alpha\}$, where $C_{\alpha}$'s are the tree of horosphere-like sets.
	
	Denote by $X^h$, the quotient space $\mathcal{G}(X,\mathcal{C})$ obtained by attaching hyperbolic cones $C_{\alpha}^h$ to $C_{\alpha} \in \mathcal{C}$ by identifying $(x,0)$ to $x$ for all $x \in C_{\alpha}$. By \thmref{thm3}, $\mathcal{G}(X,\mathcal{C})$ is a $\delta$-hyperbolic metric space for some $\delta > 0$.
	
	Recall from \cite{MjPal} that the inclusion $i_v: (X_v,\mathcal{H}_v) \to (X,\mathcal{C})$ induces a uniform proper embedding $\hat{i}_v: \widehat{X}_v \to \mathcal{TC}(X)$, i.e., for every $M > 0$, there exists $N > 0$ such that for any vertex $v \in V(\mathcal{T})$ and $x,y \in \widehat{X}_v$, $d_{\mathcal{TC}(X)}(\hat{i}_{v}(x),\hat{i}_{v}(y)) \leq M$ implies that $d_{\widehat{X}_{v}}(x,y) \leq N$.
	
	Suppose for every $v \in V(\mathcal{T})$, the inclusion map $(X_{v},\mathcal{H}_v) \to (X,\mathcal{C})$ is a proper embedding, then the induced map $i_{v} : X_{v}^h \to X^h$ is also a proper embedding.
	
	\begin{lemma}\label{lem3}\textup{\cite[Lemma 1.20]{MjPal},\cite[Lemma 2.11]{MjR}}
		Given $k, \epsilon \geq 0$, there exists $K>0$ such that if $\alpha$ and $\beta$ denote respectively a $(k,\epsilon)$- quasigeodesic in $\mathcal{TC}(X)$ and a $(k,\epsilon)$- quasigeodesic in $X^h$ joining $a$ and $b$, then $\beta \cap X$ lies in a $K$-neighbourhood of (any representative of) $\alpha$ in $(X,d)$. Here, $d$ denotes the original metric on $X$.
	\end{lemma}	
	
	\section{Cannon-Thurston maps for tree of relatively hyperbolic spaces}\label{sec5}
	
	\subsection{Ladder construction of  \cite{MjPal}}
	
	Recall that for any edge $e \in V(\mathcal{T})$ joining vertices $u$ and $v$, the maps $f_{e,u}:X_{e} \to X_{u}$  and $f_{e,v}:X_{e} \to X_{v}$ are qi-embeddings. These induce qi-embeddings $f^{h}_{e,u}:X^{h}_{e} \to X^{h}_{u}$  and $f^{h}_{e,v}:X^{h}_{e} \to X^{h}_{v}$ respectively. Let $C_2>0$ such that $f^{h}_{e,u}(X^{h}_{e})$ and $f^{h}_{e,v}(X^{h}_{e})$ are $C_2$-quasiconvex subset of $X^{h}_u$ and $X^{h}_{v}$ respectively. Let $C = C_1+C_2$, with $C_1$ from \lemref{lem7}. Let $D$ be the constant from \lemref{lem7}.
	Further, $f_{e,u}$ and $f_{e,v}$ give a partially defined map from ${X}_{u}$ to ${X}_{v}$ with the domain restricted to ${f}_{e,u}({X}_{e})$. However, we denote the map simply by ${\phi}_{u,v}: {X}_{u} \to {X}_{v}$, i.e., ${\phi}_{u,v}({f}_{e,u}(x)) = {f}_{e,v}(x)$. 
	
	We construct the ladder for geodesic rays. Recall that $p: \mathcal{TC}(X) \to \mathcal{T}$ is an  induced tree of coned-off metric spaces. Fix the vertex $v_0$ as the base point. Let $v \neq v_0$ be a vertex of $ \mathcal{T}$.
	
	\hspace{6mm} Let $\hat{\alpha}_v \subset \widehat{X}_v$ be a geodesic ray starting at a point outside the horosphere-like sets. Let $\alpha_v$ be the corresponding electro-ambient quasigeodesic ray. Consider the set of all edges incident on $v$ except for the edge lying in the geodesic joining $v_0$ to $v$ in $\mathcal{T}$. Among them, choose the collection of all edges $\{e_k\}_{k \in I}$ such that diameter of the subset $N_{C}^h(\alpha_v) \cap f_{e_k,v}(X_{e_k})$ is greater than $D$. Suppose each $e_k$ joins $v$ to $v_k \in V(\mathcal{T})$. For each $k \in I$, let $p_k$ be a nearest point projection of ${\alpha}_{v}(0)$ in $N_{C}^h(\alpha_v) \cap f_{e_k,v}(X_{e_k})$ and let $\hat{\mu}_k$ be an electric geodesic in $\widehat{X}_v$ starting at $p_k$ such that, for its electro-ambient quasigeodesic $\mu$, we have $\mu(\infty) = \alpha_{v}(\infty)$ in $\partial{X}^{h}_{v}$. Let $\widehat{\Phi}(\hat{\mu}_{k})$ denote the electric geodesic ray in  $\widehat{X}_{v_k}$, starting at $\phi_{v,v_k}(p_k)$ such that its electro-ambient quasigeodesic ray denoted by ${\Phi}(\hat{\mu}_{k})$ and the quasigeodesic ray $\phi_{v,v_k}^{h}(\mu_k)$ are asymptotic to the same point in $\partial{X}^{h}_{v_k}$. Define 
	
	$B_{1}(\hat{\alpha}) = \hat{i}_v(\hat{\alpha}) \cup \bigcup_{k} \widehat{\Phi}(\hat{\mu}_{k})$. 
	
	Now, suppose we have constructed $B_{m}(\hat{\alpha})$. Let $w_k \in p(B_{m}(\hat{\alpha})) \setminus p(B_{m-1}(\hat{\alpha}))$ and let $\hat{i}_{w_k}(\hat{\alpha_{k}}) = p^{-1}(w_k) \cap B_{m}(\hat{\alpha})$, where $\hat{\alpha}_{k}$ is a geodesic ray in $\widehat{X}_{w_k}$. So $B_{m+1}(\hat{\alpha}) = B_{m}(\hat{\alpha}) \cup \bigcup_{k}B_{1}(\hat{\alpha}_{k})$.
	The ladder $B_{\hat{\alpha}} = \cup_{m \geq 1} B_{m}(\hat{\alpha}).$
	
	Convex hull of $p(B_{\hat{\alpha}})$ is a subtree of $\mathcal{T}$ and we denote it by $\mathcal{T}_1$.
	
	\subsubsection{\textbf{Retraction map}}
	
	\begin{definitiona}\it \textbf{Retraction map:}
		{\it For each $v \in V(\mathcal{T}_1)$, let $\hat{\pi}_{\hat{\alpha}_{v}} : \widehat{X}_v \to \hat{\alpha}_v$ be the electric projection of $\widehat{X}_v$ onto $\hat{\alpha}_v$.
			
			The retraction map $\widehat{\Pi}_{\hat{\alpha}} : \mathcal{TC}(X) \to B_{\hat{\alpha}}$ is defined by:
			\[\widehat{\Pi}_{\hat{\alpha}}(x) = \hat{i}_v(\hat{\pi}_{\hat{\alpha}}(x)) \quad \text{ if } x \in \widehat{X}_v \text{ for } v \in V(\mathcal{T}_1).\]
			If $x \in p^{-1}(V(\mathcal{T})\setminus V(\mathcal{T}_1))$, we choose $x_1 \in p^{-1}(V(\mathcal{T}_1))$ such that $d(x,x_1) = d(x,p^{-1}(V(\mathcal{T}_1)))$. 
			
			Then, $\widehat{\Pi}_{\hat{\alpha}}(x) = \widehat{\Pi}_{\hat{\alpha}}(x_1)$.}
	\end{definitiona}
	
	
	\begin{lemma}\label{lem9}\textup{\cite[Lemma 3.3]{Minsky}}
		Let $Y$ be a $\delta$-hyperbolic geodesic metric space and $Z \subset Y$ a subset admitting a map $\Pi : Y \to Z$ such that there exists $C> 0$ satisfying:
		\begin{itemize}
			\item If $d(x,y) \leq 1$, then $d(\Pi(x),\Pi(y)) \leq C$;
			\item If $y \in Z$, then $d(y,\Pi(y)) \leq C$.
		\end{itemize}
		
		Then $Z$ is quasiconvex, and furthermore if $\gamma$ is a geodesic in $Y$ whose endpoints are within a distance $a$ of $Z$, then $d(x,\Pi(x)) \leq b$ for some $b = b(a,\delta, C)$ and every $x \in \gamma$. 
	\end{lemma}

	\begin{theorem}\label{thm4}
		If $\mathcal{TC}(X)$ is hyperbolic, then $B_{\hat{\alpha}}$ is uniformly quasiconvex (independent of $\hat{\alpha}$).
	\end{theorem}
	
	The retraction map is coarsely Lipschitz, i.e., there exists $C_0>0$ such that,   $d_{\mathcal{TC}(X)}(\widehat{\Pi}_{\hat{\alpha}}(x), \widehat{\Pi}_{\hat{\alpha}}(y)) \leq {C_0}d_{\mathcal{TC}(X)}(x,y) + C_0$ for every $x, y \in \mathcal{TC}(X)$. Proof of this is similar to the proof of \cite[Theorem 2.2]{MjPal}. Then, \thmref{thm4} follows from this result, along with \lemref{lem9}.

	\subsection{Vertical quasigeodesic rays}
	
	Let $\hat{\alpha}_{v}$ be an electric geodesic ray in $\widehat{X}_{v}$ starting at a point outside horospheres. Let $\alpha_v$ be its electro-ambient quasigeodesic. We have the ladder $B_{\hat{\alpha}_{v}} = \bigcup_{u \in V(\mathcal{T}_1)} \hat{i}_{u}(\hat{\alpha}_u)$. Let $B_{\alpha_{v}}^b = \bigcup_{u \in V(\mathcal{T}_1)} \hat{i}_{u}({\alpha}_{u}^b) \subset B_{\hat{\alpha}_{v}}$. For any $x \in B_{\alpha_{v}}^b$, there exists $u \in V(\mathcal{T}_1)$ such that $x \in \alpha_{u}^b$. Let $\sigma = [u_{n},u_{n-1}]\cup\cdots \cup[u_{1},u_{0}]$ be the geodesic in $\mathcal{T}_1$ with $u_0 = v$ and $u_{n} = u$.
	
	\begin{definitiona}\it \textbf{Vertical quasigeodesic ray:}
		{\it A {vertical quasigeodesic ray} starting at $x$ is a map $r_x : \sigma \to B_{\alpha_{v}}^b$ satisfying the following for a constant $C' \geq 0$:
			
			$d_{\sigma}(u,w) \leq d(r_{x}(u),r_{x}(w)) \leq C'd_{\sigma}(u,w)$,
			for all $u,w \in \sigma$.}
	\end{definitiona}
	
	\textbf{Note:} $r_{x}(u_{i}) \in X_{u_i}$ and $r_{x}(u_{n}) = x$.
	
	We end this section with one of the most important results we use.
	
	\begin{theorem}\label{thm}\textup{\cite{MjPal}} For each $v \in V(\mathcal{T})$, CT map exists for the inclusion map 
		$i_v: (X_v,\mathcal{H}_v) \to (X,\mathcal{C})$.
	\end{theorem}
	
	\section{Limit Intersection Theorem}
	
	Let $u$, $v$ be vertices connected by an edge $e$. Recall that $\phi_{u,v} : X_u \to X_v$ is a partially defined qi-embedding. By \lemref{lem5}, we know that the induced map $\phi_{u,v}^h:f_{e,u}^h(X_{e}^h) \to f_{e,v}^h(X_{e}^h)$ is a qi-embedding and it induces the embedding $\partial{\phi}_{u,v}^h:\partial{f_{e,u}^h(\partial{X_{e}^h)}} \to \partial{f_{e,v}^h(\partial{X_{e}^h)}}$ defined by $\partial{\phi}_{u,v}^h(\partial{f}_{e,u}^h(x)) = \partial{f}_{e,v}^h(x)$.
	
	\begin{definitiona}\label{flow}\it \textbf{Flow of a boundary point:}
		{\it Let $\xi \in \partial{X}_{u}^h$ and $\partial{\phi}_{u,v}^h(\xi) = \eta \in \partial{X}_v$. Then we say $\eta$ is a {\em flow} of $\xi$ and that $\xi$ can be flowed into $\partial{X}_{v}^h$.
			
			If $u_0 \neq u_n$ and $u_0, u_1,..., u_n$ is the sequence of consecutive vertices in the geodesic $[u_0,u_n]$ in $\mathcal{T}$ then we say $\xi \in \partial{X}_{u_0}^h$ can be flowed into $\partial{X}_{u_n}^h$ if there exists $\xi_i \in \partial{X}_{u_i}^h$ such that $\xi_0 = \xi$ and $\xi_{i+1} = \partial{\phi}_{u_{i},u_{i+1}}^h(\xi_{i})$ for $0 \leq i \leq n-1$. And $\xi_n$ is called a flow of $\xi$.}
	\end{definitiona}
	
	\begin{lemma}\label{lem17}
		Suppose $\xi_1 \in \partial{X}_{v_1}^h$ can be flowed to $\partial{X}_{v_2}^h$ and let $\xi_{2}$ be the flow. Then $\xi_1$ and $\xi_{2}$ map to the same limit point in $\partial{X}^h$ under the respective CT maps. 
	\end{lemma}
	
	\begin{proof}
		It is enough to prove the case when $v_1$ and $v_2$ are adjacent vertices. Let $e$ be the edge in $\mathcal{T}$ joining $v_1$ to $v_2$.
		
		By the definition of flow, $\xi_2 = \partial{\phi}_{v_1,v_2}^h(\xi_1)$. There exists $\xi_{e} \in \partial{X_e}^h$ such that $\partial{f_{e,v_i}^h}(\xi_e) = \xi_i$, for $i =1,2$. Let $\{x_n\}$ be a sequence in $X_{e}^h$ with $x_n \to \xi_e$ as $n \to \infty$. Then, for $i= 1,2$, $\{{f_{e,v_i}^h}(x_n)\}$ is a sequence in $X_{v_i}^h$ with ${f_{e,v_i}^h}(x_n) \to \xi_i$ as $n \to \infty$ and $d_{X^h}(\hat{i}_{e}(x_n),\hat{i}_{v_i}({f_{e,v_i}^h}(x_n))) = \frac{1}{2}$. This implies that  $d_{X^h}(\hat{i}_{v_1}({f_{e,v_1}^h}(x_n)),\hat{i}_{v_2}({f_{e,v_2}^h}(x_n))) = 1$. So, under CT map, both $\xi_1$ and $\xi_2$ map to the same element of $\partial {X^h}$.
	\end{proof}
	
	The converse of this lemma is false. However, we have the following:
	
	\begin{prop}\label{prop1}
		{\it Let $v_1\neq v_2 \in \mathcal{T}$. Suppose $\xi_{i} \in \partial{X}_{v_i}^h$, $i=1,2$, map to the same point, say $\xi$, under the CT maps $\partial{X}_{v_i}^h \to \partial{X}^h$ such that $\xi$ is a limit point of both $X_{v_1}$ and $X_{v_2}$. Then there exists $w \in [v_1,v_2] \subset \mathcal{T}$ such that $\xi_1$ and $\xi_2$ can be flowed to $\partial{X}_w^h$.}
	\end{prop}
	
	\begin{proof} 
		
		We assume the contrary. Suppose there exists no $w \in [v_1,v_2] \subset \mathcal{T}$ such that $\xi_1$ and $\xi_2$ can be flowed to $\partial{X}_w^h$. Then there exists $v'_{1}, v'_{2} \in [v_1,v_2]$ such that $\xi_{1}$ can be flowed only till $\partial{X}_{v'_{1}}^h$ in the direction of $X_{v_2}^h$ and $\xi_{2}$ can be flowed only till $\partial{X}_{v'_{2}}^h$ in the direction of $X_{v_1}^h$. Then, there are two possibilities.
		
		Case 1: Suppose $v'_1 \in [v'_2,v_2]$.
		
		In this case, we  are done by taking $w$ to be $v'_1$.
		
		Case 2: Suppose $v'_1 \notin [v'_2,v_2]$.
		
		We will show that this is not possible. We prove by contradiction. Using \lemref{lem17}, without loss of generality, assume $v_1 = v'_1$ and $v_2 = v'_2$. For $i=1,2$, let $\hat{\alpha}_i \subset \widehat{X}_{v_i}$ be an electric geodesic ray with corresponding electro-ambient quasigeodesic ray $\alpha_i$ such that $\alpha_i(\infty)= \xi_i$.
		Let $B_i$ denote the ladder $B_{\hat{\alpha}_i}$. Let $u_i \in V(\mathcal{T})$ be the vertex adjacent to $v_i$ in $[v_1,v_2]$, and let the edge connecting $v_i$ and $u_i$ be $e_i$. Since $\xi_i$ cannot be flowed into $\partial{X_{u_i}}$, for $i=1,2$, ${N}_C^{h}(\alpha_i)\cap f_{e_i}(X_{e_i})$ has finite diameter in $X_{v_i}^h$.
		
		Let $\{x_n\}$ be a sequence of elements in $\alpha_{1}^b$ such that $\lim{x_n} = \xi$ in $\partial{X^h}$ and $\gamma$ be a geodesic ray in $X^h$ with $\gamma(0) = x_1$ and $\gamma(\infty)= \xi$. For each $n>0$, let $y_n \in \gamma$ be a nearest point projection of $x_n$ in $\gamma$. By \lemref{lem8}, the path $\gamma|_{[x_1,y_n]}\ast [y_n,x_n]$, denoted by $\gamma_n$, is a quasi-geodesic .  Similarly we choose $\{x'_n\}$ in $\alpha_{2}^b$ with $\lim{x'_n} = \xi$ in $\partial{X^h}$. Let $\gamma'$ be a geodesic ray in $X^h$ with $\gamma'(0) = x'_1$ and $\gamma'(\infty)= \xi$. As above, for a nearest point projection $y'_n \in \gamma'$  of $x'_n$ in $\gamma'$, we get a sequence of quasi-geodesics $\gamma'_n = \gamma'|_{[x'_1,y'_n]}\ast [y'_n,x'_n]$. We will show that if Case 2 holds, then $\textup{Hd}_{X^h}(\gamma,\gamma') = \infty$, which is a contradiction.
		
		\textbf{Claim:} $\textup{Hd}(\gamma\cap X,\gamma'\cap X)= \infty.$
		
		{\em Proof of the claim:} Suppose not. Suppose there exists some $M>0$ such that $\textup{Hd}(\gamma\cap X,\gamma'\cap X)= M$. Let $\{z_k\} \subset \gamma \cap X$ and $\{z'_k\} \subset \gamma' \cap X$ such that $z_k \to \xi$, $z'_k \to \xi$ in $X^h$ and $d(z_k,z'_k) \leq M$. For each $k>0$, there exists $n_k$ such that $z_k \in \gamma_{n_k}$ and $z'_k \in \gamma'_{n_k}$. Let $\beta_{n_k}$ and $\beta'_{n_k}$ denote geodesics joining $x_1$ to $x_{n_k}$ and $x'_1$ to $x'_{n_k}$ in $B_1$ and $B_2$ respectively. By \thmref{thm4}, these are quasigeodesics in $\mathcal{TC}(X)$.
		By \lemref{lem3}, there exists $K >0$ such that $\gamma_{n_k}\bigcap X$ and $\gamma'_{n_k}\bigcap X$  lie in $K$-neighbourhood of $\beta_{n_k}$ and $\beta'_{n_k}$ respectively. So there exists $w_k \in \beta_{n_k}^b$ and $w'_k \in \beta'^{b}_{n_k}$ such that $d(z_k,w_k) \leq K$ and $d(z'_k,w'_k) \leq K$. Then, $d(w_k,w'_k)\leq M+2K = B$, say.
		
		Let $Y_1$ and $Y_2$ be the connected components obtained by removing $X_{e_1}$ from $X$, with $Y_1$ containing $X_{v_1}$ and $Y_2$ containing $X_{v_2}$. Since ${N}_C^{h}(\alpha_1)\cap f_{e_1}(X_{e_1})$ has finite diameter, only finitely many $\beta_{n_{k}}$ pass through it. So for infinitely many $k$, $w_k \in Y_1$. Since, for all such $k$, $w'_k \in Y_2$ and $d(w_k,w'_k)\leq B$, there is a sequence $\{t_k\}$ in $f_{e_1}(X_{e_1})$, and hence in $f^{h}_{e_1}(X^{h}_{e_1})$, satisfying $d(w_{k},t_{k}) \leq B$. Thus, there exists a flow of $\xi_1$ into $X_{u_1}^h$, which contradicts our assumption. This proves the claim, which further implies that $\textup{Hd}_{X^h}(\gamma,\gamma') = \infty$. 
	\end{proof}
	
	Now we show that the flow of a conical limit point is a conical limit point.
	
	\begin{lemma}\label{lem12} Let $v \in V(\mathcal{T})$ and let $\xi_v \in \partial{X}_v^h$ such that its image under the CT map, say $\xi$, is a conical limit point of $X_{v}$. Suppose $\xi_v$ can be flowed into $\partial{X}_u^h$ and let $\xi_u$ be the flow. Then $\xi_u$ also maps to a conical limit point of $X_{u}^h$ under the CT map.
	\end{lemma}
	
	\begin{proof}
		
		It is enough to check the case when $v$ and $u$ are adjacent. Rest follows by induction. So without loss of generality, assume that $d_{\mathcal{T}}(v,u) = 1$. Let $e$ be the edge in $\mathcal{{T}}$ joining $u$ to $v$. 	
		Let $\hat{\alpha}_{u}$ be an electric geodesic ray in $\widehat{X}_{u}$ with an electro-ambient quasigeodesic ray $\alpha_{u}$ satisfying $\alpha_{u}(\infty) = \xi_u$. Let $B_{\hat{\alpha}_u}$ be a ladder. Since $\xi_u$ is a flow of $\xi_v$, we have $\xi_u \in \partial{f}^h_{e,u}(\partial{X}^h_{e})$. So $f_{e,u}^h(X_{e}^h)$ is an unbounded subset of $X^h_{u}$. Let $p \in f_{e,u}^h(X_{e}^h)$ be a nearest point projection of $\alpha_{u}(0)$ on $f_{e,u}^h(X_{e}^h)$ and let ${\mu}$ be a geodesic ray in $X_{u}^h$ starting at $p$ with ${\mu}(\infty) = \xi_u$. Then, $\alpha_u$ and $\mu$ are finite Hausdroff distance apart in $X^{h}_{u}$. By quasiconvexity of $f_{e,u}^h(X_{e}^h)$, ${\mu} \subset N_{C_2}^h(f_{e,u}^h(X_{e}^h))$.  Let $x \in \alpha_{u}$ such that for a nearest point projection $y \in \mu$ of $x$ on $\mu$, $y$ satisfies $d_{X_{u}^h}(p,y)>D$, for $D>0$ from \lemref{lem7}. 
		Then by \lemref{lem7}, $[\alpha_{u}(0),p] \cup \mu|_{[p,y]} \cup [y,x] \subset N_{C_1}^h(\alpha_{u})$. Doing this for all such $x$ we have, $\mu \subset N_{C_1}^h(\alpha_{u})$. Therefore, for $C= C_1+C_2$, $N_{C}^h(\alpha_u) \cap f_{e,u}(X_{e})$ has infinite diameter in $X_{u}^h$. Hence, by the construction of $B_{\hat{\alpha}_u}$, the ladder extends to $\widehat{X}_v$ and $\hat{\alpha}_{v}= B_{\hat{\alpha}_u} \cap p^{-1}(v)$ is a geodesic ray in $\widehat{X}_u$ and for its electro-ambient quasigeodesic ray, $\alpha_{v}(\infty) = \xi_v$. 
		
		Let $\{x_n\}$ be a sequence of elements in $\alpha_{v}^b$ such that $\lim{x_n} = \xi$ in $\partial{X^h}$ and let $\gamma$ be a geodesic ray with $\gamma(0) = x_1$ and $\gamma(\infty)= \xi$. For each $n>0$, let $y_n \in \gamma$ be a nearest point projection of $x_n$ in $\gamma$. By \lemref{lem8}, $\gamma_n = \gamma|_{[x_1,y_n]}\ast [y_n,x_n]$ is a quasigeodesic ray in $X^h$. Since $\xi$ is a conical limit point of $X_{v}$, by the definition of conical limit points, there exists a real number $R\geq 0$ and an infinite sequence of elements $\{w_k\}$ in $X_v$ such that $\lim{w_k} = \xi$ and $w_k \in {N}_R^h(\gamma)$. Using \lemref{lem4}, there is $R_1 = R_1(R)$ such that, for each $k$, there exists $w'_k \in \gamma\cap X$ satisfying $d_{X^h}(w_k,w'_k) \leq d(w_k,w'_k) \leq R_1$. Let $n_k>0$ such that $w'_k \in \gamma_{n_k}$. For each $k>0$, let $\beta_{n_k}$ be a geodesic in $B_{\hat{\alpha}_u}$ joining $x_1$ to $x_{n_k}$. This  is quasigeodesic in $\mathcal{TC}(X)$.
		By \lemref{lem3}, for each $k$, there exists $z_k \in \beta_{n_k}^b$ such that $d(z_k,w'_k) \leq K$, where $K$ is the constant from \lemref{lem3}. This implies that $d(z_k,w_k) \leq K+R_1$. Since $w_k \in X_v$, we have $d_{\mathcal{T}}(v,p(z_k)) \leq K+R_1$ and $d_{\mathcal{T}}(u,p(z_k)) \leq K+R_1+1$. Then using the vertical quasigeodesic ray starting at $z_k$, we get a sequence $\{t_k\} \subset \alpha_{u}^b \subset X_u$ satisfying $d(t_k,z_k) \leq C'(K+R_1+1)$. Then 
		
		$d_{X^h}(t_k,w'_k) \leq d_{X^h}(t_k,z_k) + d_{X^h}(z_k,w'_k)  \leq  d(t_k,z_k) + d(z_k,w'_k) \leq C'(K+R_1+1)+K = L$, say.
		
		Thus, we have an infinite sequence $\{t_k\}$ in $X_u$ such that $\lim{t_k} = \xi$ in $X^h$ and $t_k \in N_{L}^h(\gamma)$. Hence, $\xi$ is a conical limit point for $X_{u}$. 
	\end{proof}
	
	This is the last lemma required to prove \thmref{thm1}.
	
	\begin{lemma}\label{lem13}
		Let $v \in V(\mathcal{T})$ and $\partial{i}_{v} : \partial{X}_{v}^h \to \partial{X}^h$ be the CT map. If $\xi \in \partial{i}_{v}(\partial{X}_{v}^h)$ is a conical limit point of $X_v$, then $|\partial{i}_{v}^{-1}(\xi)| = 1$.
	\end{lemma}
	
	\begin{proof} 
		We prove by contradiction. Let $\xi_1, \xi_2 \in \partial{X_{v}^h}$ such that $\partial{i}_{v}(\xi_1) = \partial{i}_{v}(\xi_2) = \xi \in \partial{X^h}$. For $i =1,2$, let $\hat{\alpha}_i$ be a geodesic in $\widehat{X}_{v}$ with its electro-ambient quasigeodesic $\alpha_i$ satisfying $\alpha_{i}(\infty) = \xi_i$. We follow the steps of the proof of \lemref{lem12} with respect to $\hat{\alpha}_1$ and $\hat{\alpha}_2$ to get a pair of sequences of elements that are bounded distance apart but converge to two different boundary points in $X_{v}^h$.
		
		Let $\{x_n\}$ and $\{x'_n\}$ be sequences of elements in $\alpha_{1}^b$ and $\alpha_{2}^b$ such that $\lim{x_n} = \lim{x'_n} = \xi$ in $\partial{X^h}$. Let $\gamma$ and $\gamma'$ be geodesic rays with $\gamma(0) = x_1$, $\gamma'(0) = x'_1$ and $\gamma(\infty)= \gamma'(\infty) = \xi$. So there exists $K'>0$ such that $\textup{Hd}_{X^h}(\gamma,\gamma') \leq K'$. For each $n>0$, let $y_n \in \gamma$ and $y'_n \in \gamma'$ be nearest point projection of $x_n$ on $\gamma$ and $x'_n$ on $\gamma'$ respectively. By \lemref{lem8}, $\gamma_n = \gamma|_{[x_1,y_n]}\ast [y_n,x_n]$ and $\gamma'_n = \gamma'|_{[x'_1,y'_n]}\ast [y'_n,x'_n]$ are quasi-geodesics in $X^h$. Since $\xi$ is a conical limit point of $X_{v}$, by the definition of conical limit points, there exists a real number $R\geq 0$ and an infinite sequence of elements $\{w_k\}$ in $X_v$ such that $\lim{w_k} = \xi$ and $w_k \in {N}_R^h(\gamma)$. Using \lemref{lem4}, there is $R_1 = R_{1}(R)$ and $R_2 = R_{2}(R+K')$ such that, for each $k$, there exists $z_k \in \gamma\cap X$ and $z'_k \in \gamma' \cap X$ satisfying $d(w_k,z_k) \leq R_1$ and $d(w_k,z'_k) \leq R_2$. For each $k>0$, there exists $n_k>0$ such that $z_k \in \gamma_{n_k}$ and $z'_k \in \gamma'_{n_k}$. For $i=1,2$, let $B_i = B_{\hat{\alpha}_i}$. For each $k>0$, let $\beta_{n_k}$ denote a geodesic in $B_1$ joining $x_1$ to $x_{n_k}$  and $\lambda_{n_k}$ denote a geodesic in $B_2$ joining $x'_1$ to $x'_{n_k}$. These are quasigeodesics in $\mathcal{TC}(X)$.
		
		By \lemref{lem3}, there exists a constant $K >0$ such that $\gamma_{n_k}\bigcap X$ lies in $K$-neighbourhood of $\beta_{n_k}$ and $\gamma'_{n_k}\bigcap X$ lies in $K$-neighbourhood of $\lambda_{n_k}$ in $X$. So there exists $t_k \in \beta_{n_k}^b$ and $t'_k \in \lambda_{n_k}^b$ such that $d(z_k,t_k) \leq K$ and $d(z'_k,t'_k) \leq K$. Thus, $d(w_k,t_k) \leq R_1+K$ and $d(w_k,t'_k) \leq R_2+K$. 
		Since $w_k \in X_v$, for each $k$, $d_{\mathcal{T}}(v,p(t_k))) \leq R_1+K$ and $d_{\mathcal{T}}(v,p(t'_k))) \leq R_2+K$. Using vertical quasigeodesic rays, we get sequences $\{s_k\}$ and $\{s'_k\}$ in $\alpha^b_1$ and $\alpha^b_2$ respectively, such that $d(s_k,w_k) \leq C'(R_1+K)$ and $d(s'_k,w_k) \leq C'(R_2+K)$. Then $d(s_k,s'_k) \leq C'(R_1+R_2)+2C'K$. Since $X_v \to X$ is a proper embedding, $d_{X_v}(s_k,s'_k)$ is uniformly bounded in $X_v$ and $\lim{s_k} = \lim{s'_k}$. Hence, $\xi_1 = \xi_2$.
	\end{proof}
	
	\begin{corollarya}\label{corollary}
		Let $v_1\neq v_2 \in \mathcal{T}$. Suppose $\xi_{i} \in \partial{X}_{v_i}^h$, $i=1,2$, map to the same point, say $\xi$, under the CT maps $\partial{X}^h_{v_i} \to \partial{X}^h$, such that it is a conical limit point for both $X_{v_1}$ and $X_{v_2}$, then $\xi_1$ can be flowed into $\partial{X_{v_2}}$ and $\xi_2$ can be flowed into $\partial{X_{v_1}}$.
	\end{corollarya}
	
	\subsection{Proof of \thmref{thm1}}\label{finalsec}
	
	\begin{proof}
		
		For $i=1,2$, let $w_i = g_{i}G_{v_i} \in V(\mathcal{T})$, $g_i \in G$ and $v_i \in V(Y)$.
		
		Then $G_{w_i} = \mathrm{Stab}_{G}(w_{i}) = g_{i}G_{v_i}g_{i}^{-1}$.
		
		$\Lambda_{c}(g_{i}G_{v_i}^{h}g_{i}^{-1}) = \Lambda_{c}(g_{i}G_{v_i}^{h}) = g_{i}\Lambda_{c}(G_{v_i}^{h})$.
		
		So, it is enough to show that $\Lambda_{c}(G_{v_1}^{h}) \cap \Lambda_{c}(gG_{v_2}^{h}) = \Lambda_{c}(G_{v_1}^{h} \cap gG_{v_2}^{h}g^{-1})$.
		
		It is clear that $\Lambda_{c}(G_{v_1}^{h} \cap gG_{v_2}^{h}g^{-1}) \subset \Lambda_{c}(G_{v_1}^{h}) \cap \Lambda_{c}(gG_{v_2}^{h})$ and we only need to prove \begin{equation*}
		\Lambda_{c}(G_{v_1}^{h}) \cap \Lambda_{c}(gG_{v_2}^{h}) \subset \Lambda_{c}(G_{v_1}^{h} \cap gG_{v_2}^{h}g^{-1}).
		\end{equation*}
		
		Let $\xi \in \Lambda_{c}(G_{v_1}^{h}) \cap \Lambda_{c}(gG_{v_2}^{h})$. Then there exists $\xi_{1} \in \Lambda_{c}(G_{v_1}^{h})$ and $\xi_{2} \in \Lambda_{c}(gG_{v_2}^{h})$ such that under the CT maps, $\xi_{1}, \xi_{2} \mapsto \xi$ in $\partial{G}^h$. 
		
		$\Theta_{1,v_{1}}^h : G_{v_1}^h \to X_{w_1}^h$ and $\Theta_{g,v_{2}}^h : gG_{v_2}^h \to X_{w_2}^h$ are quasi-isometries, so there exists $\xi_1' \in \partial{X_{w_1}^h}$ and $\xi_2' \in \partial{X_{w_2}^h}$ such that $\partial{\Theta}_{1,v_{1}}^h(\xi_{1}) = \xi_{1}'$ and $\partial{\Theta}_{g,v_{2}}^h(\xi_{2}) = \xi_{2}'$.
		For $i=1,2$, let $\lambda_i$ be a geodesic ray in $X_{w_i}^h$ with $\lambda_{i}(\infty) = \xi_{i}'$. By \corollaryaref{corollary}, there is a flow of $\xi_{1}'$ into $\partial{X}_{w_2}^h$ and $\xi_{2}'$ is the flow. It also follows from the proof of \lemref{lem12} that the ladder $B_{\lambda_1}$ extends to $\widehat{X}_{w_2}$ and without loss of generality, take $\lambda_2 = B_{\lambda_1} \cap \widehat{X}_{w_2}$.
		Let $\{p_k\}$ be a sequence of points on $\lambda_1$ lying outside horoball-like sets such that $\lim{p_k} = \xi_{1}'$. Let $d_{\mathcal{T}}(w_1,w_2) = N$. Then using vertical quasigeodesic rays, there exists $C' \geq 0$ and a sequence $\{q_k\}$ in $\lambda_2$, lying outside horoball-like sets, such that $\lim q_k = \xi_2'$ and $d(p_{k},q_{k}) \leq C'N$. 	For each $k>0$, let  $(\Theta_{1,v_1}^h)^{-1}(p_k) = a_k \in G_{v_1}^h$ and $(\Theta_{g,v_2}^h)^{-1}(q_k) = b_k \in gG_{v_2}^h$. Then $d(a_k,b_k) \leq d(a_k,p_k) + d(p_k,q_k) + d(q_k,b_k) \leq D_{0} + C'N + D_{0} = D'$.
		So we have sequence of points $\{a_k\}$ in $G_{v_1}$ and $\{b_k\}$ in $gG_{v_2}$ such that $d(a_k,b_k) \leq D'$ for all $k>0$, and lim $a_k = \xi_1$ and lim $b_k = \xi_2$. Let $\{\omega_k\}$ be a sequence of geodesics in the Cayley graph $\Gamma(G,S)$ joining $a_k$ to $b_k$ and let $W_k$ be a word labelling $\omega_k$. Since there are only finitely many such words, there exists a constant subsequence $\{W_{k_l}\}$ of $\{W_k\}$. Let $h_l = a^{-1}_{k_1}a_{k+l}$ and $h'_l = b^{-1}_{k_1}b_{k_l}$. Let $h \in G$ be the element represented by $W_{k_l}$. 
		Then $a_{n_1}{h_l}h = a_{n_1}h{h'_l}$, i.e., $h_l = h{h'_l}h^{-1}$. Since $h'_l$ connects two elements of $gG_{v_2}$, $h'_l \in G_{v_2}$. This implies that $h_l \in G_{v_1} \cap hG_{v_2}h^{-1}$\\
		Then $ a_{k_1}{h_l}a^{-1}_{k_1} \in a_{k_1}G_{v_1}a^{-1}_{k_1} \cap a_{k_1}hG_{v_2}h^{-1}a^{-1}_{k_1} = G_{v_1} \cap gG_{v_2}g^{-1}$.
		
		Since $d(a_{k_1}{h_l}a^{-1}_{k_1},a_{k_1}{h_l}) = d(a_{k_1}{h_l}a^{-1}_{k_1},a_{k_l}) = d(1,a_{k_1})$ for all $l \in \mathbb{N}$, $\lim_{l \to \infty}{a_{k_1}{h_l}a^{-1}_{k_1}} = \lim_{l \to \infty} {a_{k_l}} = \xi_1$. This completes the proof.
	\end{proof} 
	
	While we are far from understanding a limit intersection theorem for general limit points of vertex and edge groups of a graph of relatively hyperbolic groups satisfying the conditions of \thmref{thm1}, the following proposition sheds some light into the bounded parabolic limit points. For a finitely generated relatively hyperbolic group $G$, under the action of $G$ on $\partial{G}^h$, $g \in G$ is a {\em parabolic element} if it has infinite order and fixes exactly one point in $\partial{G}^h$. A subgroup containing only parabolic elements is a {\em parabolic subgroup} and it has a unique fixed point in the boundary. This point is called a {\em parabolic limit point}. And a parabolic limit point $p$ is {\em bounded parabolic} if its stabilizer $G_p$ in $G$ acts cocompactly on $\partial{G}^h\setminus \{p\}$.
	
	\begin{prop}\cite[Proposition 3.3]{Yang}
		Let $H$, $J$ be infinite subgroups of a relatively hyperbolic group $G$. If $\xi \in \Lambda(H) \cap \Lambda(J)$ is a bounded parabolic point of $H$ and $J$, then $\xi$ is either a bounded parabolic point of $H \cap J$, or an isolated point in $\Lambda(H) \cap \Lambda(J)$ and does not lie in $\Lambda(H \cap J)$.
	\end{prop}


\begin{thebibliography}{99}
		
		\bibitem[And94]{JWA1}
		James W. Anderson.
		\newblock {Intersections of analytically and geometrically finite subgoups of Kleinian groups}.	
		\newblock {\em Trans. Amer. Math. Soc.}, 343(1):87-98, 1994.
		
		\bibitem[And95]{JWA2}
		James W. Anderson.
		\newblock {Intersections of topologically tame subgoups of Kleinian groups}.
		\newblock {\em J. Anal. Math.}, 65:77-94, 1995.
		
		\bibitem[And96]{JWA3}
		James W. Anderson.
		\newblock {The limit set intersection theorem for finitely generated Kleinian groups.}
		\newblock {\em Math. Res. Lett.}, 3(5):675-692, 1996.
		
		\bibitem[And14]{JWA4}
		James W. Anderson.
		\newblock Limit set intersection theorems for Kleinian groups and a conjecture of Susskind.
		\newblock {\em Comput. Methods Funct. Theory}, 14(2-3):453-464, 2014.
		
		\bibitem[BesF]{BF}
		Mladen Bestvina and Mark E. Feighn.
		\newblock {A combination theorem for negatively curved groups.}
		\newblock {\em J. Differential Geom.}, 35 (1992), no. 1, 85–101.
		
		\bibitem[Bo]{Bowditch}
		B. H. Bowditch.
		\newblock Relatively hyperbolic groups.
		\newblock {\em Internat. J. Algebra Comput.}, 22(3):1250016, 66, 2012.
		
		\bibitem[BH]{MBAH}
		Martin R. Bridson and Andr\'{e} Haefliger.
		\newblock {\em Metric spaces of non-positive curvature}, volume 319 of
		{\em Grundlehren der Mathematischen Wissenschaften [Fundamental Principles of Mathematical Sciences].}
		\newblock Springer-Verlag, Berlin, 1999.
		
		\bibitem[DS]{TDDS}
		Tushar Das and David Simmons.
		\newblock {Intersecting limit sets of Kleinian subgroups and Susskind’s question.}
		
		\bibitem[DruKap]{MKCD}
		Cornelia Drutu and Michael Kapovich.
		\newblock {\em Geometric group theory}, volume 63 of American Mathematical Society Colloquium Publications. 
		\newblock {American Mathematical Society, Providence, RI, 2018.} With an appendix by Bogdan Nica.
		
		\bibitem[Farb]{Farb}
		B. Farb.
		\newblock {Relatively hyperbolic groups.}
		\newblock {\em Geom. Funct. Anal.}, 8(5):810-840, 1998.
		
		\bibitem[Grom]{Gromov}
		M. Gromov.
		\newblock {Hyperbolic Groups.} In {\em Essays in group theory}, volume 8 of {\em Math. Sci. Res. Inst. Publ.}, pages 75-263.
		\newblock Springer, New York, 1987.
		
		\bibitem[KB]{IKNB}
		Ilya Kapovich and Nadia Benakli.
		\newblock {Boundaries of hyperbolic groups}. In {\em Combinatorial and geometric group theory (New York,2000/Hoboken, NJ, 2001)} volume 296 of {\em Contemp. Math.}, pages 39-93.
		\newblock Amer. Math. Soc., Providence, RI, 2002.
		
		\bibitem[Mj]{Mj}
		Mahan Mitra.
		\newblock {Cannon-Thurston maps for trees of hyperbolic metric spaces.}
		\newblock {\em J. Differential Geom.}, 	48(1):135-164, 1998.
		
		\bibitem[Minsky]{Minsky}
		Yair Minsky.
		\newblock{Bounded geometry for Kleinian groups.}
		\newblock{\em Invent. Math.},  146 (2001), no. 1, 143–192.
		
		\bibitem[MjPal]{MjPal}
		Mahan Mj and Abhijit Pal.
		\newblock {Relative hyperbolicity, trees of spaces and Cannon-Thurston maps.}
		\newblock {\em Geom. Dedicata}, 151:59-78, 2011.
		
		\bibitem[MjR]{MjR}
		Mahan Mj and Lawrence Reeves.
		\newblock {A combination theorem for strong relative hyperbolicity.}
		\newblock {\em Geom. Topol.}, 12(3):1777-1798, 2008.
		
		\bibitem[Pal]{Pal}
		Abhijit Pal.
		\newblock {\em Cannon-Thurston maps and relative hyperbolicity.}
		\newblock {2009.}
		\newblock {Thesis (Ph.D.)-Indian Statistical Institute, Kolkata.} 
		
		\bibitem[PSar]{PS}
		Pranab Sardar.
		\newblock {Graphs of hyperbolic groups and a limit set intersection theorem.}
		\newblock {\em Proc. Amer. Math. Soc.}, 146(5):1859-1871, 2018.
		
		\bibitem[PSar1]{PS1}
		Pranab Sardar.
		\newblock {Corrigendum to "Graphs of hyperbolic groups and a limit set intersection theorem".}
		\newblock {https://arxiv.org/abs/1909.01823}.
		
		\bibitem[Serre]{Serre}
		Jean-Pierre Serre.
		\newblock {\em Trees.}
		\newblock Springer Monographs in Mathematics. Springer-Verlag, Berlin, 2003. Translated
		from the French original by John Stillwell, Corrected 2nd printing of the 1980 English translation.
		
		\bibitem[SS]{SusSwa}
		Perry Susskind and Gadde A. Swarup.
		\newblock {Limit sets of geometrically finite hyperbolic groups.}
		\newblock {\em Amer. J. Math.}, 114(2):233-250, 1992.
		
		\bibitem[Sussk]{Sus}
		Perry Douglas Susskind.
		\newblock {\em On Kleinian Groups with Intersecting Limit Sets.}
		\newblock {ProQuest LLC, Ann Arbor,
			MI, 1982. Thesis (Ph.D.)-State University of New York at Stony Brook.}
		
		\bibitem[SW]{SW}
		Peter Scott and Terry Wall.
		\newblock {\em Topological methods in group theory. in: "Homological group theory (Proc. Sympos., Durham, 1977)".}
		\newblock {pp. 137–203, London Mathematical Society Lecture Notes Series, vol. 36, Cambridge University Press, Cambridge-New York, 1979; ISBN 0-521-22729-1}
		
		\bibitem[Yang]{Yang}
		Wen-yuan Yang.
		\newblock {Limit sets of relatively hyperbolic groups}.
		\newblock {\em Geom. Dedicata}, 13:76--90, 2009.
		
		
	\end{thebibliography}
\end{document}